\documentclass{amsart}
\usepackage{amstext}
\usepackage{amsthm}
\usepackage{amssymb}
\usepackage{graphicx}
\usepackage[table]{xcolor}
\usepackage{ulem}
\usepackage{tikz-cd}
\usepackage{tikz}
\usepackage{mathdots}
\usepackage{colortbl}

\newcommand*\circled[1]{\tikz[baseline=(char.base)]{
            \node[shape=circle,draw,inner sep=2pt] (char) {#1};}}
\newtheorem{thm}{Theorem}[section]
\newtheorem{cor}[thm]{Corollary}
\newtheorem{lem}[thm]{Lemma}
\newtheorem{prop}[thm]{Proposition}

\newtheorem{conj}[thm]{Conjecture}
\theoremstyle{definition}

\theoremstyle{remark}
\newtheorem{rem}[thm]{Remark}

\numberwithin{equation}{section}
\begin{document}

\title{Guided by the primes - an exploration of very triangular numbers}
\author{Audrey Baumheckel}
\author{Tam\'as Forg\'acs}

\maketitle
\begin{quotation} \flushright{\textit{``Adding up the total of a love that's true / Multiply life by the power of two.''}} \\
\flushright{\textit{- Indigo Girls}}
\end{quotation}

\begin{abstract}
We present a string of results concerning very triangular numbers, paralleling fundamental theorems established for prime numbers throughout the history of their study, from Euclid ($\sim$300 BCE) to Green-Tao (2008) and beyond. 
\end{abstract}

\section{Introduction}
Our investigation was prompted by one of an anthology of essays that considers  a handful of engaging topics in mathematics \cite{essay}. The essay considers \textit{very triangular numbers}, which are triangular numbers whose binary representation contains a triangular number of 1s. While the formula $t_n=\frac{n(n+1)}{2}$ for the $n$th triangular number is known, to the best of our knowledge no explicit formula is known for the $n$th very triangular number. One can take the lack of such as a first parallel between the theory of very triangular numbers and the theory of prime numbers. To our delight, it turns out that there are several other results regarding prime numbers which have analogues in the theory of very triangular numbers. The ones considered in this paper (in the order they are presented) are: the infinitude of primes, arithmetic progressions and related results including the twin prime conjecture, and the existence of strings of consecutive numbers that either must, or do not contain primes. Following the sections containing our results, we conclude the paper with some open problems and potential further avenues of investigation.
In the remainder of the paper we shall use the the following notational conventions: $T$ denotes the set of triangular numbers; $VT$ denotes the set of very triangular numbers, and $|(n)_2^1|$ denotes the number of 1s in the binary representation of the positive integer $n$.

\section{The infinitude of very triangular numbers} One of the oldest results regarding the prime numbers concerns their infinitude. As is well known, Euclid's theorem (Proposition 20 in Book IX of the Elements) states that no finite list of prime numbers contains all primes.  Many of the subsequent proofs are similar to Euclid's in the sense that they establish the infinitude of primes in a non-formulaic way. There are however proofs (most recently in this \textsc{Monthly} by Elsholtz, see \cite{elsholtz}) that provide prime representing functions $f: \mathbb{N} \to \{ \textrm{prime numbers} \}$. \\
In priming the reader for further inquiry, Tanton \cite[p.217, Theorem]{essay} shows that there are infinitely many very triangular numbers, by noting that $|(t_{2^n+3})_2^1|=6 \in T$ for all $n \in \mathbb{N}$. One can readily extend his argument as follows.
\begin{thm}[Infinitude of $VT$-II] \label{thm:inftyVT} Suppose that $2(\ell+1)$ is an even triangular number, and let $n > 2\ell-1$. Then $\displaystyle{t_{2^n+2^{\ell}-1} \in VT}$. In particular, there are infinitely many very triangular numbers.
\end{thm}
\begin{proof} Note that 
\begin{eqnarray*}
\frac{(2^n+\sum_{i=0}^{\ell-1} 2^i)(2^n+2^{\ell})}{2}&=&(2^n+\sum_{i=0}^{\ell-1} 2^i)(2^{n-1}+2^{\ell-1})\\
&=&2^{2n-1}+2^{n+\ell-1}+\sum_{i=0}^{\ell-1}2^{n-1+i}+\sum_{i=0}^{\ell-1}2^{\ell-1+i}.
\end{eqnarray*}
The restriction on $n$ ensures that each power of $2$ appearing in the last expression is distinct ($2n-1>n+\ell-1>n+\ell-2>\cdots>n-1>2\ell-2>\cdots>\ell-1$). Since there are exactly $2(\ell+1)$-many distinct powers, we conclude that 
\[
\left|\left(t_{2^n+2^{\ell}-1}\right)_2^1\right|=2(\ell+1) \in T,
\]
and consequently, $\displaystyle{t_{2^n+2^{\ell}-1} \in VT}$ as claimed.
\end{proof}
Theorem \ref{thm:inftyVT} shows that given any even triangular number $k$, there are infinitely many very triangular numbers with $k$ many 1s in their binary representation. Our next result shows that given \textit{any} triangular number $k$, there are very triangular numbers with $k$ many 1s in their binary representation.
\begin{thm}[Infinitude of $VT$-III]  \label{thm:twinVText} Let $k \in T$ be a triangular number. Then $t_{2^k-2^{\ell}} \in VT$ for all $0 \leq \ell \leq \frac{k}{2}$. Moreover, $|(t_{2^k-2^{\ell}})_2^1|=k$ for all $0 \leq \ell \leq \left \lfloor \frac{k}{2} \right \rfloor$.
\end{thm}
\begin{proof} Let $k \in T$ be given. The cases $i=0,1$ are treated in Theorem \ref{thm:twinVT}. Suppose now that $2 \leq \ell \leq \left \lfloor \frac{k}{2} \right \rfloor$. We compute
\begin{eqnarray*}
t_{2^k-2^{\ell}}&=&\frac{(2^k-2^{\ell})(2^k-2^{\ell}+1)}{2}=\frac{1}{2}\left(2^{k+\ell+1}(2^{k-\ell-1}-1)+2^{\ell}(2^{k-\ell}-1) +2^{2\ell}\right)\\
&=&\sum_{j=0}^{k-\ell-2}2^{k+\ell+j}+\sum_{i=0}^{k-\ell-1}2^{i+\ell-1} +2^{2\ell-1}\\
&=&\sum_{j=0}^{k-\ell-2}2^{k+\ell+j}+\sum_{i=0}^{\ell-1}2^{i+\ell-1}+\underbrace{\sum_{i=\ell}^{k-\ell-1}2^{i+\ell-1}+2^{2\ell-1}}\\
&=&\sum_{j=0}^{k-\ell-2}2^{k+\ell+j}+\sum_{i=0}^{\ell-1}2^{i+\ell-1}+\hspace{0.5 in}2^{k-1}.
\end{eqnarray*}
Given the restriction $2 \leq \ell \leq   \left \lfloor \frac{k}{2} \right \rfloor$ we see that all powers of $2$ in the last expression are distinct. Since there are $(k-\ell-1)+\ell+1=k$ many of these, we conclude that $|(t_{2^k-2^{\ell}})_2^1|=k$, as claimed.
\end{proof}

One may wonder whether given an odd triangular number, are there very triangular numbers with that many 1s in their binary representation, and if so, how many? Setting $\ell=0$ in Theorem \ref{thm:twinVText} shows that $|(t_{2^k-1})_2^1|=(2k-2)-(k-1)+1=k$, hence the answer to the first question in in the positive. The answer to the second question -- as presented in the following two propositions -- is perhaps somewhat surprising. While there are infinitely many very triangular numbers $vt\in VT$ with $|(vt)_2^1|=15, 21, 45, 55, \ldots $(all odd triangular numbers  greater than $3$), we are reasonably certain that there are only four very triangular numbers with three 1s in their binary representations.
\begin{prop} Let $\ell >1$ so that $2 \ell+1$ is a triangular number. There are infinitely many very triangular numbers $vt \in VT$ with $| (vt)_2^1|=2\ell+1$.
\end{prop}
\begin{proof} Given $\ell$ as in the statement, consider the triangular numbers $t_{2^{2 \ell}-2^{\ell}+1}$. We compute
\begin{eqnarray*}
t_{2^{2 \ell}-2^{\ell}+1}&=&\frac{1}{2}\left(2^{2 \ell}-2^{\ell}+1 \right)\left(2^{2 \ell}-2^{\ell}+2 \right) \\
&=&2^{4 \ell-1}-2^{3 \ell}+2^{2\ell+1}-2^{\ell}-2^{\ell-1}+1 \\
&=&\sum_{i=0}^{\ell-2} 2^{3 \ell+i}+\sum_{i=0}^{\ell} 2^{\ell+i}-2^{\ell-1}+1 \\
&=&\sum_{i=0}^{\ell-2} 2^{3 \ell+i}+\sum_{i=1}^{\ell} 2^{\ell+i}+2^{\ell-1}+1.
\end{eqnarray*}
Given the restriction on $\ell$, the the powers of 2 in the above summands are all distinct, and there are $(\ell-1)+\ell+1+1=2\ell+1$ many of them.
\end{proof}
We complete this section with a discussion of triangular numbers with three 1s in their binary representations.
\begin{conj} \label{conj:no6ormore} If $n \in \mathbb{N}$ is such that $|(n)_2^1| \geq 6$, then $|(t_n)_2^1| \geq 4$.
\end{conj}
\begin{prop} The only very triangular numbers $t_n$ with $|(n)_2^1| \leq 5$ and $|(t_n)_2^1|=3$ are $21, 28, 276$ and $1540$.
\end{prop}
\begin{proof} The case of $|(n)_2^1|=1$ trivially yields $|(t_n)_2^1|=2 \neq 3$. The cases $|(n)_2^1|=2,3,4,5$ are all handled by case analysis. Lest we fall victim of paralysis by (sub)-case analysis, we only present here the case $|(n)_2^1|=3$, as this case is `big enough' to be representative of the others, and to indicate the general method of reasoning; yet small enough not to cause the reader any undue discomfort. Suppose thus that $|(n)_2^1|=3$, and write $n=2^{\sigma(1)}+2^{\sigma(2)}+2^{\sigma(3)}$ for some increasing function $\sigma: \{1,2,3\} \hookrightarrow \mathbb{N} \cup \{ 0\}$. We will need to distinguish between the cases when $\sigma(1)=0$ and when $\sigma(1)>0$. \\
Consider first the case $\sigma(1)\geq 1$, and the corresponding diagram (see Figure \ref{diag:sigma1=1}) representing the summands in $n(n+1)$, with solid arrows in the diagram pointing towards known larger quantities, and dashed ones between quantities we cannot compare without further assumptions.
\begin{figure}[ht]
\begin{center}
\begin{tikzcd}
2^{\sigma(1)} \arrow[d] \arrow[rd] \arrow[r, "+"]& \circled{$2^{\sigma(2)}$} \arrow[dl, leftrightarrow, dashed] \arrow[d] \arrow[rd] \arrow[r, "+"]& \circled{$2^{\sigma(3)}$} \arrow[dl, leftrightarrow, dashed] \arrow[d]\\ 
\textcolor{blue}{2^{2\sigma(1)}} \arrow[to=1-3, leftrightarrow, dashed] \arrow[d] \arrow[rd] \arrow[r, "+"] & \textcolor{magenta}{2^{\sigma(1)+\sigma(2)}}  \arrow[d] \arrow[rd] \arrow[r, "+"] &  \textcolor{brown}{2^{\sigma(1)+\sigma(3)}} \arrow[d]  \\ 
 \textcolor{magenta}{2^{\sigma(1)+\sigma(2)}} \arrow[d] \arrow[rd] \arrow[r, "+"] &  \textcolor{magenta}{2^{2\sigma(2)}} \arrow[d] \arrow[rd] \arrow[r, "+"]& \textcolor{brown}{2^{\sigma(2)+\sigma(3)}} \arrow[d] \\
\textcolor{brown}{2^{\sigma(1)+\sigma(3)}} \arrow[r, "+"]& \textcolor{brown}{2^{\sigma(2)+\sigma(3)}} \arrow[r, "+"]& \textcolor{brown}{2^{2\sigma(3)}}  \\
\end{tikzcd}
\caption{The case $|(n)_2^1|=3$ and $\sigma(1)\geq 1$.}
\label{diag:sigma1=1}
\end{center}
\end{figure}
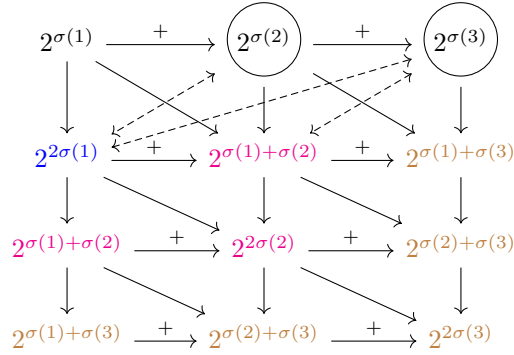

If $\sigma(2)>2\sigma(1)$, then the sum looks like $2+2^{2 \sigma(1)}+2^{\sigma(2)}+$ terms whose powers are all greater than $\sigma(2)$. As such, we conclude that $|(n(n+1))_2^1|\geq 4$.\\
If  $\sigma(2)=2 \sigma(1)$, then 
\begin{equation} \label{eq:n3one}
n(n+1)=2^{\sigma(1)}+2^{2 \sigma(1)+1}+2^{3 \sigma(1)+1}+2^{4 \sigma(1)}+ \ \textrm{terms with powers involving} \ \sigma(3).
\end{equation}
We now consider the subcases $\sigma(1)=1$ and $\sigma(1)>1$. In the former, we obtain 
\[
n(n+1)=2+2^3+2^5+ \textrm{terms with powers involving} \ \sigma(3).
\]
The smallest of the remaining powers is $\sigma(3) \geq 3$. If $\sigma(3)=4$ or $\sigma(3) \geq 6$, we trivially obtain that $|(n(n+1))_2^1| \geq4$. If $\sigma(3)=3$, we get $|(n(n+1))_2^1|=4$, while if $\sigma(3)=5$, we obtain $|(n(n+1))_2^1|=6$. If on the other hand $\sigma(1)>1$, then the four displayed powers of 2 in \eqref{eq:n3one} are all distinct. We must therefore have $\sigma(3)=2\sigma(1)+1$, or $\sigma(3)=3\sigma(1)+1$, as all other cases immediately yield $|(n(n+1))_2^1| \geq4$. Suppose first that $\sigma(3)=2\sigma(1)+1$. Then
\[
n(n+1)=2^{\sigma(1)}+2^{2 \sigma(1)+2}+2^{3 \sigma(1)+1}+2^{3 \sigma(1)+2}+2^{4 \sigma(1)+3}
\]
with all terms distinct, since $\sigma(1)>1$. Thus $|(n(n+1))_2^1|=5$. Finally, if $\sigma(3)=3\sigma(1)+1$, then 
\[
n(n+1)=2^{\sigma(1)}+2^{2 \sigma(1)+1}+2^{3 \sigma(1)+2}+2^{4 \sigma(1)}+2^{4 \sigma(1)+2}+2^{5 \sigma(1)+2}+2^{6 \sigma(1)+2},
\]
from which $|(n(n+1))_2^1| \geq 6$ readily follows.
\\
We now consider the case $\sigma(1)=0$. The corresponding diagram is given in Figure \ref{diag:sigma1=0}:

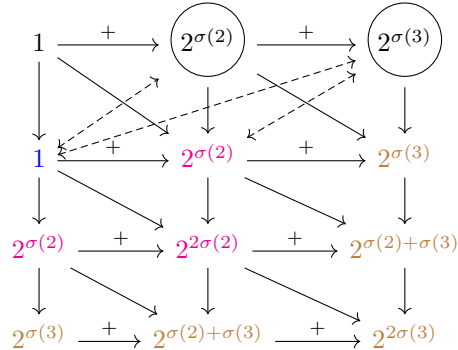
\begin{figure}[h]
\begin{center}
\begin{tikzcd}
1 \arrow[d] \arrow[rd] \arrow[r, "+"]& \circled{$2^{\sigma(2)}$} \arrow[d] \arrow[rd] \arrow[r, "+"]& \circled{$2^{\sigma(3)}$}  \arrow[d]\\ 
\textcolor{blue}{1} \arrow[ur, leftrightarrow, dashed] \arrow[to=1-3, leftrightarrow, dashed] \arrow[d] \arrow[rd] \arrow[r, "+"] & \textcolor{magenta}{2^{\sigma(2)}} \arrow [ur, leftrightarrow, dashed] \arrow[d] \arrow[rd] \arrow[r, "+"] &  \textcolor{brown}{2^{\sigma(3)}} \arrow[d]  \\ 
 \textcolor{magenta}{2^{\sigma(2)}}  \arrow[d] \arrow[rd] \arrow[r, "+"] &  \textcolor{magenta}{2^{2\sigma(2)}} \arrow[d] \arrow[rd] \arrow[r, "+"]& \textcolor{brown}{2^{\sigma(2)+\sigma(3)}} \arrow[d] \\
\textcolor{brown}{2^{\sigma(3)}} \arrow[r, "+"]& \textcolor{brown}{2^{\sigma(2)+\sigma(3)}} \arrow[r, "+"]& \textcolor{brown}{2^{2\sigma(3)}}  \\
\end{tikzcd}
\caption{The case $|(n)_2^1|=3$ and $\sigma(1)=0$.}
\label{diag:sigma1=0}
\end{center}
\end{figure}
We see that in this case
\[
(\dag) \quad n(n+1)=2+2^{\sigma(2)}+2^{\sigma(2)+1}+2^{2 \sigma(2)}+\ \textrm{terms with powers involving} \ \sigma(3).
\]
The only way for the first four summands not all to be distinct is if $\sigma(2)+1=2\sigma(2)$, or $\sigma(2)=1$. We rewrite
\[
 n(n+1)=2^2+2^3+\ \textrm{terms with powers involving} \ \sigma(3).
\]
If $\sigma(3)>3$, we obtain $|(n(n+1))_2^1| \geq 4$. If $\sigma(3)=3$, we get $n(n+1)=2^2+2^7$, hence $|(n(n+1))_2^1|=2$. Finally, if $\sigma(3)=2$, we arrive at $n(n+1)=2^3+2^4+2^5$, corresponding to the very triangular number $t_n=\frac{1}{2}n(n+1)=2^2+2^3+2^4=28$. \\
If the first four summands in $(\dag)$ are distinct, we must have $\sigma(3)=\sigma(2)+1$ if we want to reduce the number of summands. In this case however
\[
n(n+1)=2+2^{\sigma(2)}+2^{\sigma(2)+3}+2^{2 \sigma(2)}+2^{2 \sigma(2)+3},
\]
which means that $|(n(n+1))_2^1| \geq 4$ no matter what $\sigma(2)$ is. Since we have exhausted all the cases, we conclude that the only very triangular number $t_n$ with $|(n)_2^1|=3$ and $|(t_n)_2^1|=3$ is $t_7=28$. For the sake of completeness we present (without proof) the remaining very triangular numbers with three 1s in their binary representation: $|(n)_2^1|=2$ gives $t_6=21$; $|(n)_2^1|=4$ gives $t_{23}=276$, and $|(n)_2^1|=5$ yields $t_{55}=1540$. 
\end{proof}

%%%%%%%%%%%%%%%%%%%%%%%%%%%%%%%%%%%%%%%%%%%%%%%%%%%%%%%%

\section{The twin very triangular number theorem and arithmetic progressions} The twin prime conjecture has occupied the minds of many great mathematicians since its formulation, and has seen some great progress towards a positive resolution in recent years. The analogous conjecture with respect to very triangular numbers states that there are infinitely many consecutive triangular numbers which are also very triangular. Fortunately, along with being easy to state, this result is also reasonably easy to prove. We begin with two preliminary results.
 \begin{lem}
Let $k \in \mathbb{N}$. Then $2^{2k-2}+2^{2k-3}+\cdots + 2^{k-1}$ is triangular number.
\end{lem}
\begin{cor} \label{cor:infinitemanyVT}
Any number of the form $2^{2k-2}+2^{2k-3}+\cdots + 2^{k-1}, k \in T$ is very triangular.
\end{cor}
\begin{rem}
\begin{itemize}
    \item[(i)] Corollary \ref{cor:infinitemanyVT} provides an alternative proof of the Theorem on page 217 of \cite{essay} establishing the existence of infinitely many very triangular numbers.
    \item[(ii)] We also see that there is at least one very triangular number with $k$ many 1s in its binary representation for every $k \in T$.
    \item[(iii)] Not every very triangular number is of this form. For example, $(21)_2=10101$.
\end{itemize}
\end{rem}

\begin{thm}[The twin very triangular number theorem]  \label{thm:twinVT} There are infinitely many integers $n \in \mathbb{N}$ such that $t_n, t_{n+1} \in VT$.
\end{thm}
\begin{proof} Let $k \in T$, $k>1$ be given. By Corollary \ref{cor:infinitemanyVT} 
\[
\underbrace{11\cdots 1}_{k}\underbrace{00 \cdots 0}_{k-1}=\sum_{i=k-1}^{2k-2} 2^i \in VT. 
\]
Let $n+1:=\underbrace{11\cdots 1}_{k}=\sum_{i=0}^{k-1}2^i=2^k-1$. Then
\[
t_{n+1}=\frac{(n+1)(n+2)}{2}=\frac{(2^k-1)2^k}{2}=\left(\sum_{i=0}^{k-1}2^i\right)2^{k-1}=\sum_{i=0}^{k-1}2^{k-1+i}=\sum_{i=k-1}^{2k-2}2^i.
\]
Since $\displaystyle{t_{n+1}=t_n+(n+1)}$, we see that $t_{n+1}-(n+1)=(\underbrace{11\cdots 1}_{k}\underbrace{00 \cdots 0}_{k-1})-(\underbrace{11\cdots 1}_{k}) \in T$. The reader will note that
\begin{eqnarray*}
\underbrace{11\cdots 1}_{k}\underbrace{00 \cdots 0}_{k-1}-\underbrace{11\cdots 1}_{k}&=&\sum_{i=k-1}^{2k-2}2^i-\sum_{i=0}^{k-1}2^i \\
&=&\sum_{i=k-1}^{2k-2}2^i-(2^k-1)\\
&=& \sum_{i=k+1}^{2k-2}2^i+2^{k-1}+1,
\end{eqnarray*}
and that
\[
\left| \left( \sum_{i=k+1}^{2k-2}2^i+2^{k-1}+1 \right)_2^1 \right|=(k-2)+2=k\in T.
\]
We conclude that $t_n \in VT$. Since $t_{n+1} \in VT$ as well, the proof is complete.
\end{proof}
\begin{rem} The following observations are in order:
\begin{itemize}
\item[(i)] The proof of Theorem \ref{thm:twinVT} shows that if $k \in T$, then $t_{2^k-2}, t_{2^k-1} \in VT$. Not all twin pairs of very triangular numbers are of this form, however, as demonstrated by the pair $(903,946)$. 
\item[(ii)] While the theorem exhibits twin very triangular numbers $t_n, t_{n+1}$ with the property that $|(t_n)_2^1|=|(t_{n+1})_2^1|=k$ for some $k \in T$, this property does not hold in general for twin pairs. For instance, $|(t_{175})_2^1|=10$, while $|(t_{176})_2^1|=6$.
\item[(ii)] These pairs cannot be `continued' into triples of consecutive very triangular numbers, since $t_{2^k} \notin VT$ for any $k \in \mathbb{N}$. 
\end{itemize}
\end{rem}

Theorem \ref{thm:twinVT} raises the question of just how long strings of consecutive triangular numbers one can have that are also very triangular. The reader can easily check that $|(t_{581})_2^1|=|(t_{582})_2^1|=|(t_{583})_2^1|=10$ and that$|(t_{1702})_2^1|=|(t_{1703})_2^1|=|(t_{1704})_2^1|=|(t_{1705})_2^1|=10$ (there are many more such examples) which would indicate that perhaps one can find arbitrarily long strings of consecutive triangular numbers that are also very triangular. More generally, one may look for very triangular numbers in arithmetic progressions. In the `prime-realm' we are thus led to the celebrated Green-Tao theorem concerning such progressions.

\begin{thm}[Green-Tao \cite{gt}] Let $\pi(N)$ denote the number of primes less than or equal to $N$. If $A$ is a subset of the prime numbers such that 
\[
\limsup_{N \to \infty} \frac{|A \cap[1,N]|}{\pi(N)}>0,
\]
then for all positive $k$, the set $A$ contains infinitely many arithmetic progressions of length $k$. In particular, the entire set of prime numbers contains arbitrarily long arithmetic progressions.
\end{thm}
 We formulate the analogous conjecture regarding very triangular numbers as follows.
\begin{conj} \label{conj:vtap} For any positive integer $k$, there exists an arithmetic progression $AP(k)$ of length $k$ so that $t_j \in VT$ for all $j \in AP(k)$. 
\end{conj}
A plausible approach to proving this conjecture could start with the following theorem of Szemr\'edi.
\begin{thm}[Szemer\'edi, \cite{sz75}] \label{thm:szemeredi} If $A \subset \mathbb{N}$ is such that
\[
\limsup_{N \to \infty} \frac{|A \cap \{1,2,\ldots,N\}|}{N}>0,
\]
then for all $k \in \mathbb{N}$, $A$ contains infinitely many arithmetic progressions of length $k$.
\end{thm}
If $\pi_{VT}(x)$ and $\pi_T(x)$ denote the counting functions of very triangular and triangular numbers less than or equal to $x$, then
\[
\frac{\pi_{VT}(t_N)}{\pi_T(t_N)}=\frac{|VT \cap \{t_1,t_2,\ldots, t_N \}|}{N}=\frac{| \sigma(\mathbb{N}) \cap \{1,2,\ldots, N\}|}{N},
\]
where $\{ t_{\sigma(n)} \}_{n=1}^{\infty}$ is the sequence of very triangular numbers. Consequently, by applying Theorem \ref{thm:szemeredi} we see that if $\limsup_{N \to \infty} \frac{\pi_{VT}(t_N)}{\pi_T(t_N)} >0$, then $\sigma(\mathbb{N})$ contains infinitely many arithmetic progressions of length $k$ for every $k \in \mathbb{N}$. For any such progression $a\cdot j+b$, $0 \leq j \leq k-1$, the set $\{ t_{a\cdot j+b} \}_{j=0}^{k-1}$ is a subset of $VT$. Our CAS computations -- limited by available memory -- are inconclusive in terms of the potential value of the critical quantity $\limsup_{N \to \infty} \frac{\pi_{VT}(t_N)}{\pi_T(t_N)}$. Undeterred, and following historical precedence\footnote{Before proving the full result, Szemer\'edi first established the existence of infinitely many arithmetic progressions with lengths $k=4$ (see \cite{sz69})}, we next demonstrate the existence of infinitely many arithmetic progressions $AP(k)$ or length six or less, such that $\{t_j\}_{j \in AP(k)} \subset VT$\footnote{When it comes to prime numbers, the terminology \textit{primes in arithmetic progression} refers to consecutive terms in an arithmetic progression which are all prime. \textit{Consecutive primes in arithmetic progression} on the other hand refers to consecutive terms in an arithmetic progression which are prime, with only composite numbers between them.}. In fact, we establish the existence of strings of consecutive very triangular numbers which are also consecutive triangular numbers. The only result we need in order to be able to demonstrate our claim is the following proposition.
\begin{prop}\label{prop:repeatingpattern} Suppose that $n >5$. Then for all $m \geq n$ and $0 \leq k <2^{\lfloor (n-1)/2 \rfloor}$, the following equality holds:
\[
\left |(t_{2^m+3+k})_2^1\right|=\left |(t_{2^n+3+k})_2^1\right|.
\] 
\end{prop}
\begin{proof} Note first that for any $n$ and $0\leq k$,
\[
t_{2^n+3+k}=2^{2n-1}+2^{n+1}+(k+1)2^n+2^{n-1}+4+2+t_k+3k.
\]
Given the restriction $k < \lfloor (n-1)/2 \rfloor$, we see that 
\begin{align*}
t_k&=\frac{k(k+1)}{2}<\frac{2^{\lfloor (n-1)/2 \rfloor}(2^{\lfloor (n-1)/2 \rfloor}+1)}{2} \leq 2^{n-2}+2^{\lfloor (n-1)/2 \rfloor-1}< 2^{n-2}+2^{n-4},\\
&\textrm{and}\\
3k&< 3 \cdot 2^{\lfloor (n-1)/2 \rfloor}=2^{\lfloor (n-1)/2 \rfloor+1}+2^{\lfloor (n-1)/2 \rfloor}<2^{n-2}+2^{n-3}.
\end{align*}
It follows that $t_k+3k <2^{n-1}$, and hence
\begin{align*}
\left |(t_{2^m+3+k})_2^1\right|&= \left|(2^{2n-1}+2^{n+1}+(k+1)2^n+2^{n-1})_2^1 \right|+\left|(t_k+3k+4+2)_2^1 \right| \\
&=\left|(2^{2m-1}+2^{m+1}+(k+1)2^m+2^{m-1})_2^1 \right|+\left|(t_k+3k+4+2)_2^1 \right|\\
&=\left |(t_{2^n+3+k})_2^1\right|.
\end{align*}
The proof is complete.
\end{proof}
\begin{cor} \label{cor:6cons} There are infinitely many $j \in \mathbb{N}$ such that $\{t_{j+\ell}\}_{\ell=0}^5 \subset VT$.
\end{cor}
\begin{proof} In light of Proposition \ref{prop:repeatingpattern}, it suffices to find an $n >5$ and integers $0 \leq k<k+1<k+2<k+3<k+4<k+5 <2^{\lfloor (n-1)/2 \rfloor}$ so that $\{ t_{2^n+3+k+\ell} \}_{\ell=0}^5 \subset VT$. Using a CAS, we find that $n=30$, and $k=1870<1871<1872<1873<1874<1875<2^{\lfloor (30-1)/2 \rfloor}=2^{14}$ is such a pair. In fact
\begin{eqnarray*}
&& \left|(t_{2^{30}+1873})_2^1 \right| =\left|(t_{2^{30}+1874})_2^1 \right|=\left|(t_{2^{30}+1875})_2^1 \right|= \\
&& \left|(t_{2^{30}+1876})_2^1 \right|= \left|(t_{2^{30}+1877})_2^1 \right|=\left|(t_{2^{30}+1878})_2^1 \right|=21.
\end{eqnarray*}
\end{proof}
We remark that the first instance of six consecutive triangular numbers that are all very triangular occurs much sooner: $\{ t_{30301},t_{30302},t_{30303},t_{30304},t_{30305},t_{30306}\} \subset VT$. Corollary \ref{cor:6cons} generalizes the twin very triangular number theorem, and could be viewed as the sextuplet very triangular number theorem, although we suspect that such a moniker is not likely to catch on. 

\section{Gaps and consecutive strings} Results concerning strings of consecutive natural numbers that either must contain primes, or do not contain primes have a simple elegance about them. Bertrand's postulate (see \cite{bertrand}) states that for all integers $n >3$ there exists a prime $p$ satisfying $n<p<2n-2$. On the other hand, considering the integers $(n+1)!, (n+1)!+2, (n+1)!+3, \ldots, (n+1)!+(n+1)$ we see that given any $n \in \mathbb{N}$, there are infinitely many strings of $n$ or more consecutive composite numbers. Much to our delight, we found results analogous to these concerning very triangular numbers. 
\begin{thm}[Bertrand's postulate for very triangular numbers] For $n=4,5,6$ and $n>9$, there exists a very triangular number $m \in VT$ with $t_n<m<t_{2n}$.
\end{thm}
\begin{proof} Suppose first that $n >9$. Then there exists $k \in \mathbb{N}$ such that either 
\begin{itemize}
\item[(i)] $n=2^{k-1}+1$, 
\item[(ii)] $n=2^{k-1}+2$, or 
\item[(iii)]$2^{k-1}+2<n\leq 2^k$. 
\end{itemize}
It's easy to check that in cases (i) and (ii), $n< 2^{k-1}+3< 2n$, and consequently $t_n <t_{2^{k-1}+3} <t_{2n}$. Since $n>9$, we see that $k>4$, and by Theorem \ref{thm:inftyVT} $t_{2^{k-1}+3} \in VT$. In case (iii), we have $n\leq 2^k <2^k+3<2^k+4<2n$, and once more we conclude that $t_n<t_{2^k+3}<t_{2n}$, with $t_{2^k+3} \in VT$. Table \ref{Bertrand} completes the proof by checking the cases for $n=4,5,6$.

\begin{table}[htp]
\caption{Bertrand's postulate for very triangular numbers (small $n$)}
\begin{center}
\begin{tabular}{c|ccc}
$n$ & $t_n$ & $m \in VT$ & $t_{2n}$\\
\hline
\hline
4 & 10 & 21,28 & 36 \\
5 & 15 & 21,28 & 55 \\
6 & 21 & 28 & 78 \\
7 & 28 & N/A& 105 \\
8 & 26 & N/A & 136 \\
9 & 45 & N/A & 171\\
10 & 55 & 190 & 210\\
$\vdots$ & $\vdots$ & $\vdots$ & $\vdots$ \\
19 & 210 & 231, 276, 378, 435, 630 & 741\\
$\vdots$ & $\vdots$ & $\vdots$ & $\vdots$ \\
\hline
\end{tabular}
\end{center}
\label{Bertrand}
\end{table}%

\end{proof} 
Next, we turn our attention to the gaps between consecutive very triangular numbers.

\begin{prop} \label{prop:gaps} Suppose that $k \in T$ and that $4 \mid k$. If $2^k-2^{\frac{k}{2}} < n \leq 2^k-2^{ \frac{k}{2}}+\frac{k}{4}$, then $t_n \notin VT$.
\end{prop}
\begin{proof} Let $n, k=4 \ell$ be as in the statement, and write $n=2^k-2^{\frac{k}{2}}+m$. Suppose first that $m=2^p$ for some $0 \leq p \leq \lfloor \log_2 \ell \rfloor$. In this case $t_n$ can be written as
\begin{eqnarray*}
t_n&=&\frac{1}{2}\left(\sum_{j=2\ell}^{4\ell-1}2^j+2^p \right)\left(\sum_{j=2\ell}^{4\ell-1}2^j+2^p+1 \right)\\
&=&\frac{1}{2}\left(\left( \sum_{j=2\ell}^{4\ell-1}2^j\right)^2+\sum_{j=2\ell+p+1}^{4\ell+p}2^{j} + \sum_{j=2\ell}^{4\ell-1}2^j+2^{2p}+ 2^p\right) \\
&=&\frac{1}{2}\left(\sum_{j=6\ell+1}^{8 \ell -1} 2^j +2^{4 \ell}+\sum_{j=2\ell+p+1}^{4\ell+p}2^{j} + \sum_{j=2\ell}^{4\ell-1}2^j+2^{2p}+ 2^p \right)\\
&=&\frac{1}{2}\left(\sum_{j=6\ell+1}^{8 \ell -1} 2^j +2^{4 \ell+p+1}+\sum_{j=2\ell+p+1}^{4\ell-1}2^{j} + \sum_{j=2\ell}^{4\ell-1}2^j+2^{2p}+ 2^p \right)\\
&=&\frac{1}{2}\left(\sum_{j=6\ell+1}^{8 \ell -1} 2^j +2^{4 \ell+p+1}+\sum_{j=2\ell+p+2}^{4\ell}2^{j} + \sum_{j=2\ell}^{2\ell+p}2^j+2^{2p}+ 2^p \right).
\end{eqnarray*}
Given that $0 \leq p \leq \lfloor \log_2 \ell \rfloor$, we see that 
\[
p\leq 2p<2\ell \leq 2\ell+p<2\ell+p+2<4\ell<4\ell+p+1<6\ell+1
\]
as soon as $\ell \geq 2$, and hence
\[
|(t_n)_2^1|=\begin{cases} 4\ell+1=k+1 &\mbox{if } p=0 \\
4\ell+2=k+2 & \mbox{if } p \geq 1 \end{cases} \notin T \quad \mbox{ for any} \quad k \in T \quad \mbox{with} \quad k\geq 3.
\]
Finally, if $\ell=1$, then $n=13$ and $|(t_{13})_2^1|=5 \notin T$. \\
 Assume now that $n, k=4 \ell$ are as in the statement, and write $n=2^k-2^{\frac{k}{2}}+m$, with 
\[
m=\sum_{j=1}^{r}  2^{\sigma(j)} \quad \textrm{for some} \quad 2 \leq r \leq \lfloor \log_2 \ell \rfloor,
\] 
and 
\[
\sigma:\{1,2,\ldots, r\} \hookrightarrow \{0,1,2,\ldots,\lfloor \log_2 \ell \rfloor\}.
\]
A simple calculation shows that
\[
t_n=\frac{1}{2}\left(2^{k/2}+\left(\sum_{j=k/2+1}^k 2^j+\sum_{j=1}^r 2^{\sigma(j)}\right)\left(\sum_{j=1}^r 2^{\sigma(j)}+1 \right) +\sum_{j=3k/2+1}^{2k-1} 2^j\right).
\]
The reader will note that
\[
2 \leq \sum_{j=1}^r 2^{\sigma(j)}+1 \leq 2^{\sigma(r)+1} \leq 2^{r+1} \leq 2^{\log_2(\ell)+1} \leq \frac{k}{2}<2^{k/2},
\]
and that
\[
2 \sigma(r)+1 \leq 2r+1 \leq 2 \log_2(\ell)+1<2\ell=\frac{k}{2}.
\]
Consequently, the largest power of 2 in the expression
\[
\left(\sum_{j=k/2+1}^k 2^j \right)\left(\sum_{j=1}^r 2^{\sigma(j)}+1 \right)
\]
is less than or equal to $3k/2$, whereas the largest power of 2 in the expression
\[
\left(\sum_{j=1}^r 2^{\sigma(j)}\right)\left(\sum_{j=1}^r 2^{\sigma(j)}+1 \right)
\]
is less than $k/2$. Therefore, 
\begin{eqnarray} \label{eq:gapcount}
|(t_n)_2^1|&=&\left|\left(\left(\sum_{j=k/2+1}^k 2^j+\sum_{j=1}^r 2^{\sigma(j)}\right)\left(\sum_{j=1}^r 2^{\sigma(j)}+1 \right) \right)_2^1 \right|+\frac{k}{2} \nonumber \\
&=&\left| \left(\left(\sum_{j=k/2+1}^k 2^j\right)\left(\sum_{j=1}^r 2^{\sigma(j)}+1 \right)\right)_2^1\right|\\
&+&\left|\left(\left(\sum_{j=1}^r 2^{\sigma(j)}\right)\left(\sum_{j=1}^r 2^{\sigma(j)}+1 \right)\right)_2^1 \right| +\frac{k}{2} \nonumber
\end{eqnarray}
We compute the first summand in \eqref{eq:gapcount}. To this end, note that
\begin{eqnarray*}
&& \left(\sum_{j=k/2+1}^k 2^j\right)\left(\sum_{j=1}^r 2^{\sigma(j)}+1 \right)=\sum_{j=k/2+1}^k 2^j+\sum_{j=1}^r \left(2^{\sigma(j)+k+1}-2^{\sigma(j)+k/2+1}\right)\\
&=&\sum_{j=k/2+1}^k 2^j+\sum_{i=1}^{\textcolor{red}{r}} \sum_{j=\sigma(i)+k/2+1}^{k+\sigma(i)} 2^j\\
&=&\left(\sum_{\textcolor{red}{j=\sigma(r)+k/2+1}}^{k+\sigma(r)} 2^j\right)+\sum_{j=k/2+1}^k 2^j+\sum_{i=1}^{r-1} \sum_{j=\sigma(i)+k/2+1}^{k+\sigma(i)} 2^j\\
&=&2^{k+1}+\sum_{j=k/2+1}^{\sigma(r)+k/2} 2^j +\left(\sum_{j=\sigma(r)+k/2+2}^{k+\sigma(r)} 2^j\right)+\sum_{i=1}^{\textcolor{red}{r-1}} \sum_{j=\sigma(i)+k/2+1}^{k+\sigma(i)} 2^j\\
&=&2^{k+1}+2^{k/2+\sigma(r)+1} +\sum_{j=k/2+1}^{\sigma(r-1)+k/2} 2^j +\left(\sum_{j=\sigma(r)+k/2+2}^{k+\sigma(r)} 2^j+\sum_{j=\sigma(r-1)+k/2+2}^{k+\sigma(r-1)} 2^j\right)\\
&+&\sum_{i=1}^{\textcolor{red}{r-2}} \sum_{j=\sigma(i)+k/2+1}^{k+\sigma(i)} 2^j\\
&=&2^{k+1}+2^{k/2+\sigma(r)+1}+2^{k/2+\sigma(r-1)+1} +\sum_{j=k/2+1}^{\sigma(r-2)+k/2} 2^j \\
&+&\left(\sum_{j=\sigma(r)+k/2+2}^{k+\sigma(r)} 2^j+\sum_{j=\sigma(r-1)+k/2+2}^{k+\sigma(r-1)} 2^j+\sum_{j=\sigma(r-2)+k/2+2}^{k+\sigma(r-2)} 2^j\right)\\
&+&\sum_{i=1}^{r-3} \sum_{j=\sigma(i)+k/2+1}^{k+\sigma(i)} 2^j\\
&=&\textcolor{blue}{2^{k+1}}+2^{k/2+\sigma(r)+1}+2^{k/2+\sigma(r-1)+1}+\cdots+2^{k/2+\sigma(2)+1} +\sum_{j=k/2+1}^{\sigma(1)+k/2} 2^j \\
&+&\left(\textcolor{blue}{\sum_{j=\sigma(r)+k/2+2}^{k+\sigma(r)} 2^j}+\sum_{j=\sigma(r-1)+k/2+2}^{k+\sigma(r-1)} 2^j+\sum_{j=\sigma(r-2)+k/2+2}^{k+\sigma(r-2)} 2^j+\cdots+\sum_{j=\sigma(1)+k/2+2}^{k+\sigma(1)} 2^j\right)\\
&=&\textcolor{blue}{2^{k+\sigma(1)+1}}+2^{k/2+\sigma(r)+1}+2^{k/2+\sigma(r-1)+1}+\cdots+2^{k/2+\sigma(2)+1} +\sum_{j=k/2+1}^{\sigma(1)+k/2} 2^j \\
&+&\left(\sum_{j=\sigma(r)+k/2+2}^{k+\sigma(r)} 2^j+\cdots+\sum_{j=\sigma(3)+k/2+2}^{k+\sigma(3)} 2^j+\textcolor{blue}{\sum_{j=\sigma(2)+k/2+2}^{k+\sigma(2)} 2^j}+\sum_{j=\sigma(1)+k/2+2}^{k} 2^j\right)\\
&=&\textcolor{blue}{2^{k+\sigma(2)+1}}+2^{k/2+\sigma(r)+1}+2^{k/2+\sigma(r-1)+1}+\cdots+2^{k/2+\sigma(2)+1} +\sum_{j=k/2+1}^{\sigma(1)+k/2} 2^j \\
&+&\left(\sum_{j=\sigma(r)+k/2+2}^{k+\sigma(r)} 2^j+\cdots+\textcolor{blue}{\sum_{j=\sigma(3)+k/2+2}^{k+\sigma(3)} 2^j}+\sum_{j=\sigma(2)+k/2+2}^{k+\sigma(1)} 2^j+\sum_{j=\sigma(1)+k/2+2}^{k} 2^j\right)
\end{eqnarray*}
\begin{eqnarray*}
&=&2^{k+\sigma(r)+1}+\textcolor{blue}{2^{k/2+\sigma(r)+1}}+2^{k/2+\sigma(r-1)+1}+\cdots+\textcolor{magenta}{2^{k/2+\sigma(2)+1}} +\sum_{j=k/2+1}^{\sigma(1)+k/2} 2^j \\
&+&\left(\sum_{j=\sigma(r)+k/2+2}^{k+\sigma(r-1)} 2^j+\textcolor{blue}{\sum_{j=\sigma(r-1)+k/2+2}^{k+\sigma(r-2)} 2^j}+\cdots+\sum_{j=\sigma(2)+k/2+2}^{k+\sigma(1)} 2^j+\textcolor{magenta}{\sum_{j=\sigma(1)+k/2+2}^{k} 2^j}\right)\\
&=&2^{k+\sigma(r)+1}+\sum_{j=2}^r 2^{k/2+\sigma(j)+1}+ \sum_{j=k/2+1}^{\sigma(1)+k/2} 2^j +\sum_{j=2}^r \left(\sum_{i=\sigma(j)+k/2+2}^{k+\sigma(j-1)} 2^i\right)+\sum_{j=\sigma(1)+k/2+2}^k 2^j\\
&=&2^{k+\sigma(r)+1}+\sum_{j=\sigma(r)+k/2+2}^{k+\sigma(r-1)} 2^j+\left(\sum_{j=3}^r 2^{k/2+\sigma(j)+1}+\sum_{j=2}^{r-1}\left(\sum_{i=\sigma(j)+k/2+2}^{k+\sigma(j-1)} 2^i \right) \right)\\
&+&2^{k/2+\sigma(2)+1}+\sum_{j=k/2+1}^{\sigma(1)+k/2} 2^j+\sum_{j=\sigma(1)+k/2+2}^k 2^j\\
&=&2^{k+\sigma(r)+1}+\textcolor{blue}{\sum_{j=\sigma(r)+k/2+2}^{k+\sigma(r-1)} 2^j}+\left(\sum_{j=3}^r 2^{k/2+\sigma(j)+1}+\sum_{j=2}^{r-1}\left(\sum_{i=\sigma(j)+k/2+2}^{k+\sigma(j-1)} 2^i \right) \right)\\
&+&\textcolor{blue}{2^{k+1}}+\sum_{j=k/2+1}^{\sigma(1)+k/2} 2^j+\sum_{j=\sigma(1)+k/2+2}^{k/2+\sigma(2)} 2^j\\
&=&2^{k+\sigma(r)+1}+2^{k+\sigma(r-1)+1}+\sum_{j=\sigma(r)+k/2+2}^{k} 2^j+\left(\sum_{j=3}^r 2^{k/2+\sigma(j)+1}+\sum_{j=2}^{r-1}\left(\sum_{i=\sigma(j)+k/2+2}^{k+\sigma(j-1)} 2^i \right) \right)\\
&+&\sum_{j=k/2+1}^{\sigma(1)+k/2} 2^j+\sum_{j=\sigma(1)+k/2+2}^{k/2+\sigma(2)} 2^j.
\end{eqnarray*}
Now,
\begin{eqnarray*}
&&\sum_{j=3}^r 2^{k/2+\sigma(j)+1}+\sum_{j=2}^{r-1}\left(\sum_{i=\sigma(j)+k/2+2}^{k+\sigma(j-1)} 2^i \right)\\
&=&\sum_{j=3}^r 2^{k/2+\sigma(j)+1}+\sum_{j=3}^{r}\left(\sum_{i=\sigma(j-1)+k/2+2}^{k+\sigma(j-2)} 2^i \right)\\
&=&\sum_{j=3}^r\left(2^{k+\sigma(j-2)+1}+\sum_{i=\sigma(j-1)+k/2+2}^{k/2+\sigma(j)} 2^i \right),
\end{eqnarray*}
hence, we arrive at the expression
\begin{eqnarray*}
\sum_{j=1}^r 2^{k+\sigma(j)+1}+\sum_{j=\sigma(r)+k/2+2}^{k} 2^j+\sum_{j=2}^r\left(\sum_{i=\sigma(j-1)+k/2+2}^{k/2+\sigma(j)} 2^i \right)+\sum_{j=k/2+1}^{\sigma(1)+k/2} 2^j.
\end{eqnarray*}
Therefore,
\begin{eqnarray*}
&&\left| \left(\left(\sum_{j=k/2+1}^k 2^j\right)\left(\sum_{j=1}^r 2^{\sigma(j)}+1 \right)\right)_2^1\right|\\
&=&r+(k-(\sigma(r)+k/2+2)+1)+\sum_{j=2}^r (k/2+\sigma(j)-(\sigma(j-1)+k/2+2)+1)\\
&+&(\sigma(1)+k/2-(k/2+1)+1\\
&=&r+k/2-\sigma(r)-1+\sum_{j=2}^r (\sigma(j)-\sigma(j-1))-(r-1)+\sigma(1)\\
&=&k/2-\sigma(r)+(\sigma(r)-\sigma(1))+\sigma(1)=k/2.
\end{eqnarray*}
We now turn our attention to the second summand in \eqref{eq:gapcount}. Since
\[
\sum_{j=1}^r 2^{\sigma(j)}+1\leq \sum_{j=0}^{\lfloor \log_2 \ell \rfloor} 2^{j}+1\leq2^{\lfloor \log_2 \ell \rfloor +1},
\]
we see that
\[
\left(\sum_{j=1}^r 2^{\sigma(j)}\right)\left(\sum_{j=1}^r 2^{\sigma(j)}+1 \right) \leq \sum_{j=1}^r 2^{\sigma(j)+\lfloor \log_2 \ell \rfloor+1}<2^{2(\lfloor \log_2 \ell \rfloor+1)}.
\]
Consequently,
\begin{eqnarray*}
1 &\leq& \left|\left(\left(\sum_{j=1}^r 2^{\sigma(j)}\right)\left(\sum_{j=1}^r 2^{\sigma(j)}+1 \right)\right)_2^1 \right| \\
&\leq& 2 \lfloor \log_2 \ell \rfloor+1 \leq 2(\log_2 k -\log_2 4)+1=2 \log_2 k-3.
\end{eqnarray*}
Recall that $k \in T$. If we write $k=\frac{s(s+1)}{2}$ for some $s \in \mathbb{N}$, then the smallest triangular number greater than $k$ is $\frac{s(s+1)}{2}+(s+1)$. On the other hand, 
\begin{eqnarray*}
2 \log_2k-3&=&2 \log_2\left(\frac{s(s+1)}{2} \right)-3=2(\log_2 s + \log_2 (s+1))-5\\
&<&4 \log_2 (s+1)-5\stackrel{\star}{<}s+1,
\end{eqnarray*}
where the starred inequality is equivalent to $(s+1)^4<2^{s+6}$, which one easily establishes by induction. We conclude that if $t_n \in T$ with $n=2^k-2^{k/2}+m$ for $1 \leq m \leq  k/4 $, then 
\[
k=\frac{s(s+1)}{2} < |(t_n)_2^1| \leq \frac{k}{2}+\frac{k}{2}+2\log_2k-3 <\frac{s(s+1)}{2}+(s+1),
\] and consequently $t_n \notin VT$. The proof is complete.
\end{proof}
As a corollary to Proposition \ref{prop:gaps} we immediately obtain the following
\begin{thm}[The gap theorem for very triangular numbers]\label{thm:gaps} For any $k \in \mathbb{N}$ there exist consecutive very triangular numbers with at least $k$ triangular numbers between them.
\end{thm}

\section{Future work} We invite the reader to carry out investigations similar to ours, either to make our understanding of very triangular numbers more complete, or to extend these results to other sets of integers (or both!) One may start by proving Conjecture \ref{conj:no6ormore}. Using techniques similar to those presented in this paper we were able to show that if $|(n)_2^1|=6$, then $|(t_n)_2^1| \geq4$. We suspect however that this approach will not suffice in settling the conjecture for all $|(n)_2^1| \geq7$. \\
Another beautiful result yet to be proved is Conjecture \ref{conj:vtap}. As mentioned in the paper, numerical evidence is inconclusive regarding the density measure that we might use (along with Szemer\'edi's result). Nonetheless, our sense (based on numerical results) is that very triangular numbers are abundant, and hence should have positive upper density in the set of triangular numbers. Also related to arithmetic progressions is the question whether or not one can partition the set $VT$ with finitely many arithmetic progressions, and finally, whether it is possible to find a closed formula for the $n$th very triangular number. Finding such would answer many of the posed questions in one fell swoop.    \\
Finally, a possible generalization of the present work concerns subsets of the natural numbers $f: \mathbb{N} \hookrightarrow \mathbb{N}$ with the property that $s \in S \subset f(\mathbb{N})$ implies that $|(t)_2^1| \in f(\mathbb{N})$, and how the properties of $f$ effect whether or not $S$ exhibits properties analogous to those of $VT \subset T$.

%\begin{acknowledgment}{Acknowledgment.}
%The authors wish to thank the Greek polymath Anonymous, whose prolific works are an endless source of inspiration.
%\end{acknowledgment}

\end{document}